\documentclass[preprint]{imsart} 

\RequirePackage{amsthm,amsmath,amsfonts,amssymb}
\RequirePackage[authoryear]{natbib} 
\RequirePackage[colorlinks,breaklinks,citecolor=blue,urlcolor=blue]{hyperref}
\RequirePackage{graphicx}

\usepackage{breakcites}

\usepackage[normalem]{ulem}
\usepackage{enumerate}

\startlocaldefs
\theoremstyle{plain}

\newtheorem{theorem}{Theorem}
\newtheorem{lemma}[theorem]{Lemma}
\newtheorem{corollary}[theorem]{Corollary}
\newtheorem{prop}[theorem]{Proposition}
\theoremstyle{remark}

\newcommand{\R}{\mathbb{R}}

\newcommand{\EE}{\mathbb{E}}
\newcommand{\E}{\mathbb{E}}

\newcommand{\PP}{\mathbb{P}}

\renewcommand{\>}{\rangle}
\newcommand{\<}{\langle}

\newcommand{\III}[1]{{\left\vert\kern-0.25ex\left\vert\kern-0.25ex\left\vert #1 
    \right\vert\kern-0.25ex\right\vert\kern-0.25ex\right\vert}}

\newcommand{\simiid}{\stackrel{\text{iid}}{\sim}}

\newcommand{\ca}{\mathcal{A}} \newcommand{\cb}{\mathcal{B}}  \newcommand{\cd}{\mathcal{D}_+^{p\times p}}  \newcommand{\cf}{\mathcal{F}}        \newcommand{\cn}{\mathcal{N}}        \newcommand{\cw}{\mathcal{W}}   

\newcommand{\eps}{\varepsilon}

\makeatletter
\providecommand*{\diff}%
	{\@ifnextchar^{\DIfF}{\DIfF^{}}}
\def\DIfF^#1{%
	\mathop{\mathrm{\mathstrut d}}%
		\nolimits^{#1}\gobblespace}
\def\gobblespace{%
	\futurelet\diffarg\opspace}
\def\opspace{%
	\let\DiffSpace\!%
	\ifx\diffarg(%
		\let\DiffSpace\relax
	\else
		\ifx\diffarg[%
			\let\DiffSpace\relax
	\else
		\ifx\diffarg\{%
			\let\DiffSpace\relax
		\fi\fi\fi\DiffSpace}

\newcommand{\TT}{\mathsf{T}} 

\newcommand{\acov}{\Sigma}
\newcommand{\aprec}{\Theta}
\newcommand{\acor}{\Gamma}
\newcommand{\apcor}{\Omega}

\newcommand{\truecov}{\acov^*} 
\newcommand{\trueprec}{\aprec^*} 
\newcommand{\truecor}{\acor^*} 
\newcommand{\truepcor}{\apcor^*} 

\newcommand{\estcov}{\widehat\acov}
\newcommand{\estprec}{\widehat\aprec}

\newcommand{\estpcor}{\widehat\apcor}

\def\argmin{\mathop{\rm argmin}}
\def\argmax{\mathop{\rm argmax}}
\newcommand{\psd}{{\mathcal{S}_{\succeq 0}^{p\times p}}}

\newcommand{\M}{{\mathcal{M}^{p\times p}}}

\newcommand{\mtptwo}{\text{MTP}_2}

\newcommand{\Ldiv}{L^{\text{ssym}}}

\newcommand{\vertiii}[1]{{\vert\kern-0.25ex\vert\kern-0.25ex\vert #1 
    \vert\kern-0.25ex\vert\kern-0.25ex\vert}}

\endlocaldefs

\begin{document}

\begin{frontmatter}
\title{Covariance Estimation with\\Nonnegative Partial Correlations}
\runauthor{Soloff, Guntuboyina \& Jordan}
\runtitle{Cov. Est. with Nonnegative Partial Correlations}

\begin{aug}
\author[A]{\fnms{Jake A.} \snm{Soloff}\ead[label=e1]{soloff@berkeley.edu}\thanksref{t1}},
\author[A]{\fnms{Adityanand} \snm{Guntuboyina}\ead[label=e2]{aditya@stat.berkeley.edu}\thanksref{t2}}
\& 
\author[B]{\fnms{Michael I.} \snm{Jordan}\ead[label=e3]{jordan@cs.berkeley.edu}\thanksref{t1}}
\address[A]{Department of Statistics, University of California, Berkeley,
\printead{e1,e2}}

\address[B]{Department of Electrical Engineering \& Computer Sciences, Department of Statistics, University of California, Berkeley,
\printead{e3}}

\thankstext{t1}{Supported by the Mathematical Data Science program of the Office of Naval Research under Grant N00014-18-1-2764}
\thankstext{t2}{Supported by NSF CAREER Grant DMS-1654589}
\end{aug}

\begin{abstract}
We study the problem of high-dimensional covariance estimation
under the constraint that the partial correlations are nonnegative. 
The sign constraints dramatically simplify estimation: the Gaussian 
maximum likelihood estimator is 
well defined with only two observations regardless of the number of variables. We analyze its performance 
in the setting where the dimension may be much larger than the sample size. We establish that the estimator is 
both high-dimensionally consistent and minimax optimal in the 
symmetrized Stein loss. We also prove a negative result which shows that the sign-constraints can introduce substantial bias for estimating the top eigenvalue of the covariance matrix. 

\end{abstract}



\end{frontmatter}

\section{Introduction}\label{sec-intro}
Consider the problem of estimating a $p \times p$ covariance matrix $\truecov$ and
its inverse $\trueprec := (\truecov)^{-1}$ from an $n\times p$ data matrix $X$ whose rows are independently distributed according to the multivariate normal distribution $\cn(0, \truecov)$ with mean zero and covariance matrix $\truecov$. 
The maximum likelihood
estimator (MLE) of $\trueprec$ is given by
\begin{equation}\label{fullmle}
 \tilde{\aprec} := \argmin_{\aprec \in \psd} \left\{\<\aprec, S\>
   - \log \det \aprec \right\},
\end{equation}
where $\psd$ denotes the set of all $p \times p$ symmetric, positive
semi-definite (PSD) matrices, $\<\aprec, S\> := \text{tr}(\aprec^\TT S)$ denotes the Frobenius inner product, and $S$ is the sample covariance
matrix, defined as 
\begin{equation}\label{samcova}
  S := n^{-1}X^\TT X. 
\end{equation}
It is well known that $\tilde{\aprec}$ exists if and only if $S$ is
nonsingular, in which case $\tilde\aprec = S^{-1}$. 
In particular, in the high-dimensional setting where $p > n$, the MLE
does not exist, since the minimum in~\eqref{fullmle} is not finite. \citet{slawski2015estimation} observed, however, that if the
optimizer in \eqref{fullmle} is constrained to lie in the set of $p \times p$ positive semidefinite matrices with \textit{nonpositive
off-diagonal entries}, then, with probability one, the optimum is
well-defined and attained for all $n \geq 2$ regardless of the value of
$p$. Specifically, let 
\begin{equation*}
  \M := \left\{\aprec\in \psd : \aprec_{jk} \le 0 \text{ for } j\ne k\right\},
\end{equation*}
and observe that it is the convex cone of symmetric \textit{$M$-matrices}, an  important class of matrices appearing in many contexts \citep[see, e.g.,][Chap. 6]{berman1994nonnegative}. \citet{slawski2015estimation} proved that the optimizer
\begin{equation}\label{eq-log-det-div-min}
  \estprec := \argmin_{\aprec \in \M} \left\{\<\aprec, S\>
   - \log \det \aprec \right\},
\end{equation}
exists uniquely as long as, in the observed sample, no two variables
are perfectly positively correlated (i.e., $S_{jk} < \sqrt{S_{jj} S_{kk}}$ for
all $j \neq k$) and no variable is constant (i.e., $S_{jj} > 0$ for
all $j$). Both conditions hold with probability one
under the assumed Gaussian model for $n 
\geq 2$, and thus, unlike the unconstrained MLE in~\eqref{fullmle}, the
estimator~\eqref{eq-log-det-div-min} is well-defined even in the
high-dimensional regime. 

The constrained MLE $\estprec$ presents an elegant, tuning-free method for estimating
precision matrices which works for $n \ge 2$ and all values of $p$
under the assumption $\trueprec\in \M$. Efficient
algorithms for computing $\estprec$ are given in \citet{slawski2015estimation} and \citet{lauritzen2019maximum}. Note that the precision matrix
having nonpositive off-diagonal entries $\trueprec_{jk}$ is equivalent to nonnegative partial correlations $-\trueprec_{jk}/\sqrt{\trueprec_{jj}\trueprec_{kk}}$  \citep{bolviken1982probability}. Examples
of practical covariance estimation problems with nonnegative partial
correlations abound~\citep[see, e.g.,][]{lake2010discovering, slawski2015estimation, agrawal2019covariance}. More generally, \citet{karlin1983m} showed that for the normal distribution the condition that the precision matrix belongs to $\M$ is equivalent to multivariate total positivity of
order two ($\mtptwo$). $\mtptwo$ is a strong form of positive
dependence \citep{colangelo2005some} that has been widely used in
auction theory \citep{milgrom1982theory}, actuarial sciences \citep{denuit2006actuarial}, and educational evaluation and policy analysis \citep{chade2014student}.

There is growing interest in $\estprec$ in the graph signal processing literature \citep{pavez2016generalized, egilmez2017graph, pavez2018learning}, where $M$-matrices are known as {\sl Generalized Graph Laplacians} (GGL). Indeed, every graph Laplacian is a diagonally dominant $M$-matrix, and conversely every $M$-matrix $\aprec\in \M$ can be viewed as a generalized graph Laplacian, in the sense that it has a sparse {\sl edge-incidence factorization} $\aprec = VV^\TT$, where $V\in \R^{p\times p(p+1)/2}$ has at most two nonzero entries per column, whereas  positive semidefinite matrices that have other sign patterns typically require dense factorizations \citep{boman2005factor}. This connection to nonnegative weighted graphs has led to a host of other application areas in image processing and network analysis.


This paper studies the statistical properties of $\estprec$ as an
estimator of the unknown precision matrix $\trueprec$ in the high-dimensional regime. Even though
$\estprec$ exists uniquely for all $n \ge 2$ regardless of the
value of $p$, rigorous results have not yet been proved for the accuracy of
$\estprec$ in the high-dimensional regime. In the classical low dimensional asymptotic regime where $p$ is fixed and $n \to \infty$, \citet{slawski2015estimation} apply standard results for $M$-estimators to show
consistency of $\estprec$. More recently, \cite{lauritzen2019maximum} provide an elegant perspective on $\estprec$ and a bound on the support graph $G(\estprec) = \{(j, k) : \estprec_{jk} < 0\}$, and \cite{wang2019learning} develop a consistent estimator of $G(\trueprec)$. 

The study of consistency
and optimality properties of $\estprec$ requires fixing an
appropriate loss function. Because $\estprec$ is defined via
maximum likelihood, it is natural to work with the
\textit{Stein loss}: 
\begin{align}\label{eq-stein-def}
L^\text{s}(\aprec, \trueprec) 
&:=  \frac{1}{p}\<\aprec, \truecov\> - \frac{1}{p}\log\det \aprec\truecov - 1,
\end{align}
which, up to scaling by $p$, is the Kullback-Leibler divergence between
multivariate mean zero normal distributions with precision matrices 
$\aprec$ and $\trueprec$ respectively. The Stein loss has a long
history of application in covariance matrix estimation \citep{james1961, stein1975estimation, stein1986lectures, dey1985estimation, ledoit2018optimal, donoho2018}. In this paper, we work with the 
\textit{symmetrized Stein loss} (alternatively known as the
\textit{divergence loss}), defined as
\begin{equation}\label{eq-divergence-loss-defn}
\begin{aligned}
\Ldiv(\aprec, \trueprec) 
&:= \frac{L^{\text{s}}(\aprec, \trueprec) + L^{\text{s}}(\trueprec,
  \aprec)}{2} 
= \frac{1}{2p}\left\<\aprec -  \trueprec, \truecov-\acov\right\>,
\end{aligned}
\end{equation}
where $\acov=\aprec^{-1}$. Note that $\Ldiv(\aprec, \trueprec)$ is symmetric and $2\Ldiv(\aprec, \trueprec)$ clearly dominates both the Stein
loss and the reversed Stein loss $L^{\text{s}}(\trueprec,
\aprec)$ (which is also known as the \textit{entropy loss}). Properties of $\Ldiv$ are further discussed in Section~\ref{sec-stein}. 

We use the $1/p$ scaling in the loss function \eqref{eq-divergence-loss-defn} because, as explained by \citet{ledoit2018optimal}, this is necessary for consistency in the high-dimensional regime where the number of variables $p$ may be much larger than the sample size $n$. Indeed, in the simple case where
$\aprec^*$ is known to be diagonal, the natural estimator is the
diagonal matrix $\hat{\aprec}^{\mathrm{DIAG}}$ with diagonal entries
$1/S_{jj}, {j=1,\dots,p}$ (where $S$ is the sample covariance matrix
defined in \eqref{samcova}). It is easy to see that
$\left\<\hat{\aprec}^{\mathrm{DIAG}} -  \trueprec, \truecov-\hat{\acov}^{\mathrm{DIAG}}\right\>$ is of the order $p/n$
which will be far from zero in the high-dimensional regime where $p >
n$. 

We present results on the performance of $\estprec$ in the symmetrized Stein loss in Section \ref{sec-stein}. Our main result in Theorem~\ref{thm-symm-kl} implies that $\Ldiv(\estprec,
\trueprec)$ converges to zero as long as $\log p = o(n)$. This
implies high-dimensional consistency of $\estprec$. Moreover, the rate of convergence is $\sqrt{\frac{\log p}{n}}$, which we prove in Theorem \ref{thm-minimax-rate} is
optimal in the minimax sense. Thus $\estprec$ is minimax optimal
in the high-dimensional regime under the symmetrized Stein loss. Our results provide rigorous support for the
assertion that the nonpositive off-diagonal constraint provides
strong implicit regularization in the high-dimensional
regime. 
In Theorem \ref{thm-instance-lower-bound}, we also lower bound the loss $\Ldiv(\estprec, \trueprec)$ which implies that the $\sqrt{n}$ rate is not an artifact of our analysis even when the true precision matrix $\trueprec$ is diagonal. 

High-dimensional consistency with the rate $\sqrt{\frac{\log p}{n}}$ has
appeared previously in many papers on covariance and precision matrix
estimation---see for instance  \cite{rothman2008sparse, yuan2010high, ravikumar2011high, cai2011constrained, sun2013sparse} and \cite{cai2016structure} for a review of rates in structured covariance estimation. Most of these results are
for estimators that use explicit regularizers (such as the $\ell_1$
penalty in the Graphical Lasso \cite{banerjee2008model, friedman2008sparse, mazumder2012graphical}),
which is crucially exploited by the proof techniques and assumptions employed in these
papers. By contrast, the regularization induced by the assumption
$\trueprec \in \M$ is implicit and we consequently use different arguments relying on
careful use of the KKT conditions underlying the optimization~\eqref{eq-log-det-div-min}. Our analysis identifies a bound relating the entries of an $M$-matrix to its spectrum, providing new insight into the simplifying structure of the convex cone $\M$.

The symmetrized Stein loss has the additional symmetry property of invariance under inversion:
$\Ldiv(\estcov, \truecov) = \Ldiv(\estprec, \trueprec)$  where
$\estcov := \estprec^{-1}$. This means that $\hat{\Sigma}$
is also a high-dimensionally-consistent estimator of $\truecov$. The
choice of the loss function is quite crucial here. In
Section~\ref{sec-spectral}, using the Perron-Frobenius theorem and a careful analysis of the entry-wise positive part $S_+$ of the sample covariance, we prove a negative result which shows that, for the maximum eigenvalue, $\estcov$ can be much worse as an
estimator of $\truecov$ compared to the sample covariance matrix
$S$. This result indicates that enforcing the sign-constraints can exacerbate bias in the estimation of the top eigenvalue. 

The paper is organized as follows: Section~\ref{sec-stein} contains our main results establishing optimality of $\estprec$, Section~\ref{sec-spectral} establishes suboptimality under the spectral norm, and Section~\ref{sec-discussion} has a discussion which touches upon some related issues including misspecification (where $\trueprec \not\in \M$), estimation of correlation matrices and connections to shape-restricted regression. Finally Section \ref{sec-proofs} contains proofs of all the results of the paper. 

\section{Symmetrized Stein Loss: Consistency and Optimality}\label{sec-stein}
This section contains our results on the high-dimensional consistency and optimality of $\estprec$ under the symmetrized Stein loss $\Ldiv$ defined in \eqref{eq-divergence-loss-defn}. We start by describing some basic properties of $\Ldiv$. 

The expected value of the objective in \eqref{eq-log-det-div-min}, $\<\aprec, \truecov\> - \log\det \aprec$, agrees up to factors depending only on $\truecov$ with the {\sl Stein loss}~\eqref{eq-stein-def}, which is also a matrix Bregman divergence \citep{dhillon2008matrix}, proportional to the Kullback-Leibler (KL) divergence between centered multivariate Gaussian distributions: $\frac{2}{p}D(\cn(0, \Sigma)\|\cn(0,\truecov))$. It is well known that the KL divergence is not symmetric. When the inputs to the divergence are reversed, the resulting Bregman divergence is also known as the {\sl entropy loss}, $L^\text{ent}(\aprec, \trueprec) := L^\text{s}(\trueprec, \aprec)$. The sum of these loss functions dominates each, and conveniently does not directly involve any determinants. Following \citet{ledoit2018optimal}, we define $\Ldiv = \frac{L^\text{s} + L^\text{ent}}{2}$ to be the average of the two loss functions. Commonly known as the {\sl symmetrized Stein loss} or {\sl divergence loss}, $\Ldiv$ is equal to the \citet{jeffreys1946invariant} divergence between two centered multivariate Gaussian distributions, divided by $p$. Definition~\eqref{eq-divergence-loss-defn} entails a number of useful and important properties for the symmetrized Stein loss:
\begin{enumerate}[(i)]
\item (Nonnegativity) $\Ldiv(\aprec, \aprec^*) \ge 0$, with equality if and only if $\aprec = \aprec^*$. 
\item (Symmetry) $\Ldiv(\aprec, \aprec^*) = \Ldiv(\aprec^*, \aprec)$.
\item (Invariance under inversion) $\Ldiv(\aprec, \aprec^*) = \Ldiv(\acov, \acov^*)$.
\item\label{property-congruence} (Invariance under congruent transformations) For all $p\times p$ nonsingular matrices $P$, we have the scale-invariance property:
\begin{equation}\label{eq-scale-invariance}
    \Ldiv(\aprec, \aprec^*) = \Ldiv(P^\TT\aprec P, P^\TT\aprec^*P)
\end{equation}
\end{enumerate}
The symmetrized Stein loss thus induces a natural geometry on the space of PSD matrices---see \cite{moakher2006symmetric} for a review and comparison to other geometries. We emphasize that triangle inequality fails to hold for both $\Ldiv$ and $\sqrt{\Ldiv}$. As a loss, $\Ldiv$ treats the dual problems of estimating the covariance matrix and the precision matrix equally. It can also be shown that the symmetrized Stein loss is equivalent to the squared Frobenius norm when the input matrices $\aprec$ and $\aprec^*$ have bounded spectra.

In terms of the eigenvalues $(\lambda_j)_{j=1}^p$ of $\aprec\truecov$, the symmetrized Stein loss is simply the goodness-of-fit measure
\begin{align}\label{eq-goodness-of-fit}
\Ldiv(\aprec, \trueprec)
= \frac{1}{p}\sum_{j=1}^p\frac{(\lambda_j-1)^2}{2\lambda_j}.
\end{align}
This alternative representation provides further insight into the normalization of the loss~\eqref{eq-divergence-loss-defn} with a factor of $p$. The symmetrized Stein loss is the expectation of the function $\lambda\mapsto \frac{(\lambda-1)^2}{2\lambda}$ with respect to the empirical spectral distribution of $\aprec\truecov$. This expectation measures how far the spectrum of $\aprec\truecov$ deviates from a point mass at one, which is the spectrum of the identity $I_p$. In asymptotic settings where $p=p(n)\to \infty$ as $n\to \infty$, a natural consistency criterion checks whether this expectation converges to zero. 

Our analysis of the symmetrized Stein loss $\Ldiv(\estprec, \trueprec)$ involves the maximum population correlation between any two variables: 
\begin{equation*}
    \max_{j \neq k} \frac{\Sigma^*_{jk}}{\sqrt{\Sigma^*_{jj} \Sigma^*_{kk}}}. 
\end{equation*}
We assume that the above quantity is strictly less than 1 which is clearly necessary for $\Sigma^*$ to be nonsingular i.e., for $\Theta^*$ to exist. Our bound on $\Ldiv(\estprec, \trueprec)$ will involve the quantity: 
\begin{align*}
    \gamma(\truecov) := \left(1-\max_{j\ne k} \frac{\truecov_{jk}}{\sqrt{\truecov_{jj}\truecov_{kk}}}\right)^{-1}.
\end{align*}
It is natural for $\gamma(\truecov)$ to enter the analysis in light of the existence result of \citet{slawski2015estimation} which states that the maximum sample correlation must be less than one in order for the estimator $\estprec$ to be well-defined. Note that $\gamma(\Sigma^*)$ is the smallest $\gamma \geq 1$ such that 
\begin{equation}\label{eq-gamma-def}
    \max_{j \neq k} \frac{\Sigma^*_{jk}}{\sqrt{\Sigma^*_{jj} \Sigma^*_{kk}}} \leq 1 - \gamma^{-1} < 1. 
\end{equation}
Because $\gamma(\Sigma^*)$ is defined in terms of population correlations, it is scale-invariant. Note that $\Ldiv$ also has this scale invariance property (see  \eqref{eq-scale-invariance}).

\begin{theorem}\label{thm-symm-kl} Let $S = n^{-1}X^\TT X$ denote the sample covariance matrix based on data matrix $X\in \R^{n\times p}$ with i.i.d.\ $\cn(0, \truecov)$ rows, where $\trueprec = (\truecov)^{-1}\in \M$. For all $n\ge c_1\gamma^2(\truecov)\log p$, the MLE $\estprec$ defined in~\eqref{eq-log-det-div-min} satisfies
\begin{align}\label{eq-div-thm}
\Ldiv(\estprec, \trueprec)
&\le c_2\gamma(\truecov)\sqrt{\frac{\log p}{n}},
\end{align}
with probability at least $1-c_3p^{-2}$. Here $c_1,c_2,c_3$ are universal positive constants.
\end{theorem}

Theorem~\ref{thm-symm-kl} states that $\estprec$ is high-dimensionally consistent in the symmetrized Stein loss $\Ldiv$ as long as $\log p = o(n)$. We prove Theorem~\ref{thm-symm-kl} in Section~\ref{sec-proofs}, deriving a basic inequality from the first order optimality conditions for~\eqref{eq-log-det-div-min} and showing that concentration of the intrinsic noise $\|S-\truecov\|_\infty$ is sufficient to control the basic inequality. Crucially, we use the fact that every $M$-matrix $\aprec\in \M$ is up to diagonal scaling equivalent to a diagonally dominant matrix \citep[see][Chap. 6, Property $M_{34}$]{berman1994nonnegative}. 

We emphasize that the result holds without additional assumptions on the underlying precision matrix such as sparsity. Consistency in the symmetrized Stein loss is a strong guarantee compared to the recent literature on optimal shrinkage of the sample covariance $S$ under high-dimensional asymptotics \citep{donoho2018, ledoit2018optimal}, where the symmetrized Stein loss $\Ldiv$ converges to a nonzero limit under the asymptotic regime $p/n\to \alpha > 0$ as $n\to \infty$. By contrast, for the constrained MLE the loss $\Ldiv(\estprec, \trueprec)$ converges in probability to zero whenever $\log p=o(n)$.


Since the upper bound \eqref{eq-div-thm} depends only on the true precision matrix $\trueprec$ through the population quantity $\gamma(\truecov)$, Theorem~\ref{thm-symm-kl} actually bounds the worst case risk obtained from the divergence  loss over all $M$-matrices $\trueprec$ with $\gamma(\truecov)$ bounded. It is natural to question whether the $\sqrt{n}$ rate is improvable. Our next result shows that, in the high-dimensional setting where $p$ grows superlinearly in $n$, the minimax rate over the class of $M$-matrices with $\gamma(\truecov)\le \gamma$ matches the $\sqrt{\frac{\log p}{n}}$ rate from Theorem~\ref{thm-symm-kl}. 

\begin{theorem}\label{thm-minimax-rate} Let $X\in\R^{n\times p}$ have i.i.d.\ $\cn(0, \truecov)$ rows, and suppose the number of variables $p$ satisfies $c_1n^\beta\le p\le \exp(c_2 n)$. For every $\gamma > 1$, we have
\begin{align}\label{eq-minimax}
\inf_{\breve\aprec = \breve\aprec(X)}\sup_{\substack{\trueprec\in \M \\ \gamma(\truecov) \le \gamma}} \E\Ldiv(\breve\aprec, \trueprec)
\ge c_\gamma\sqrt{\frac{\log p}{n}}.
\end{align}
 Here $c_1,c_2 >0$ and $\beta > 1$ are universal constants and $c_\gamma > 0$ is a constant depending only on $\gamma$.
\end{theorem}

Paired with Theorem~\ref{thm-symm-kl}, this result implies that $\estprec$ is minimax optimal in the symmetrized Stein loss over $M$-matrices with correlations bounded away from one. Our proof adapts the construction of \citet{cai2016estimating}, Theorem 4.1, which lower bounds the minimax risk in the spectral norm over a parameter set of sparse precision matrices of the form $I + \eps A$, where $\eps$ depends on problem parameters $p$ and $n$, and $A$ is an adjacency matrix. A key aspect of this approach is to allow for different perturbations over the rows and columns of $A$, in order to recover the $\sqrt{n}$ rate \citep{kimobtaining}. 

The $M$-matrix constraint provides implicit regularization and is crucial for achieving the minimax rate $\sqrt{\frac{\log p}{n}}$. If this constraint is dropped, it is impossible for any estimator to achieve a rate better than $\sqrt{\frac{p}{n}}$ when $p > n$. This follows from the next result where we prove a minimax lower bound of $\sqrt{\frac{p}{n}}$ for the $\Ldiv$ loss function over the entire class $\psd$ of positive semidefinite matrices when $p > n$. On the other hand, $\M$ is much larger than diagonal matrices because the minimax rate of estimation over the class $\cd$ of positive diagonal matrices in the $\Ldiv$ loss function is $1/n$ (this is also proved in the next result). In summary, the class of $M$-matrices acts as a strong high-dimensional regularizer while being considerably larger than the class of all positive diagonal matrices.  

\begin{prop}\label{prop-minimaxdiagpsd}
Fix $p$ and $n > 2$. The minimax risk in the symmetrized Stein loss over diagonal precision matrices satisfies
\begin{align}\label{eq-minimax-diag}
\inf_{\hat\aprec = \hat\aprec(X)}\sup_{\trueprec\in \cd} \E\Ldiv(\hat\aprec, \trueprec)
\asymp \frac{1}{n}.
\end{align}
The minimax risk in the symmetrized Stein loss over PSD matrices satisfies
\begin{align}\label{eq-minimax-psd}
\inf_{\hat\aprec = \hat\aprec(S)}\sup_{\trueprec\in \psd} \E\Ldiv(\hat\aprec, \trueprec)
\gtrsim \min\left\{\frac{p}{n}, \sqrt{\frac{p}{n}}\right\}.
\end{align}
\end{prop}


Theorem \ref{thm-minimax-rate} implies that the $\sqrt{n}$ rate of Theorem \ref{thm-symm-kl} cannot be improved in worst case over the entire class $\M$. In the next result, we prove that the $\sqrt{n}$ rate for $\estprec$ cannot be improved even when the truth $\Theta^*$ lies in the class $\cd$ of positive diagonal matrices. In other words, this shows that $\estprec$ does not adapt to the minimax rate over $\cd$. 


\begin{theorem}\label{thm-instance-lower-bound} Suppose $\trueprec\in \cd$ is a positive diagonal matrix and $c_1p \ge \sqrt{n}$. Then 
\begin{align}
\Ldiv(\estprec, \trueprec) &\ge \frac{c_1}{2\sqrt{n}},
\end{align}
with probability at least $1-3p\exp\left(-c_2(n\land p)\right)$, where $c_1$ and $c_2$ are universal positive constants.
\end{theorem}

\section{Spectral Norm: Suboptimality}\label{sec-spectral}
In this section, we prove a negative result which implies that $\estprec$ and $\estcov$ can be suboptimal for estimating spectral quantities of $\trueprec$ and $\truecov$ respectively. Consider the case when $\truecov = I_p$ and consider estimation of the top eigenvalue $\lambda_{\max}(\truecov) = 1$. The performance of the sample covariance matrix $S$ is well understood. Indeed, in the asymptotic setting $p/n \rightarrow \alpha > 0$, \citet{geman1980limit} proved that 
\begin{align*}
\lambda_{\max}(S)
\to (1+\sqrt{\alpha})^2,
\end{align*}
in probability as $n\to \infty$. This implies that $S$ is inconsistent for the estimation of $\lambda_{\max}(\truecov)$ when $p/n$ converges to a positive constant. Our next result proves that $\estcov = \estprec^{-1}$ is also inconsistent for estimating $\lambda_{\max}(\truecov)$ and, more interestingly, its performance is \textit{substantially worse} compared to $S$. Specifically, in the same asymptotic setting where $p/n \rightarrow \alpha > 0$, we have
\begin{equation}\label{eq-muchworse}
    \lambda_{\max}(\estcov) \rightarrow \infty
\end{equation}
in probability as $n \rightarrow \infty$. Thus the introduction of the sign constraints make the resulting covariance matrix estimator $\estcov$ much worse compared to $S$ for estimating the principal eigenvalue. This should be contrasted with the high-dimensional minimax optimality results from the previous section in the symmetrized Stein loss. 



\begin{theorem}\label{thm-inconsistency} Suppose $\truecov = I_p$ and $p\ge 17$. Then 
\begin{align}\label{eq-lambdamax}
\lambda_{\max}(\estcov)\ge 1+c_1 \frac{p}{\sqrt{n}},
\end{align}
with probability at least $1-3p\exp\left(-c_2(n\land p)\right)$, for some universal positive constants $c_1, c_2$. 
\end{theorem}

Note that when $p/n \rightarrow \alpha > 0$, the right hand side of \eqref{eq-lambdamax} diverges to $\infty$ which proves \eqref{eq-muchworse}. 

The proof of Theorem \ref{thm-inconsistency} is crucially based on following dual formulation to the constrained MLE~\eqref{eq-log-det-div-min} \citep[see, e.g.,][]{slawski2015estimation}:
\begin{align}\label{eq-dual}
\estcov = \argmax_{\substack{\Sigma\in \psd \\ \Sigma\ge S,\, D_\Sigma = D_S}} \det \Sigma,
\end{align}
where the second constraint $\Sigma\ge S$ is an entry-wise inequality. This fact and the well-known observation that the inverse of an $M$-matrix is entry-wise nonnegative \citep[see][Chap. 6, Property $N_{38}$]{berman1994nonnegative} together imply that $\estcov_{jk}\ge S_{jk}\lor 0$ for all $j, k$. This allows us to prove Theorem \ref{thm-inconsistency} by a careful analysis of the entry-wise positive part matrix $S_+$ of $S$. 



Theorem \ref{thm-inconsistency} implies minimax suboptimality of $\estcov$ in the spectral norm $\vertiii{\cdot}_2$. To see this, note that, for every $K > 0$, the sample covariance $S$ satisfies the worst case risk bound
\[
\sup_{\substack{\truecov\in \psd \\ \lambda_{\max}(\truecov) \le K}} \E\vertiii{S-\truecov}_2
\le CK\left(\sqrt{\frac{p}{n}} + \frac{p}{n}\right),
\]
where $C > 0$ is a universal constant \citep[see, e.g.,][Example 6.3]{wainwright2019high}. By contrast, Theorem~\ref{thm-inconsistency} implies 
\[
\sup_{\substack{\trueprec\in \M \\ \lambda_{\max}(\truecov)\le K}}\E\vertiii{\estcov - \truecov}_2
\ge \E_{\truecov=KI_p}\left[\lambda_{\max}(\estcov) - K\right]
\ge cK\frac{p}{\sqrt{n}}
\]
for $n \gtrsim \log p$. Hence $\estcov$ is minimax suboptimal in the spectral norm for most choices of $p$ and $n$.

Theorem~\ref{thm-inconsistency} also implies inconsistency in spectral norm for the precision matrix. Since $\lambda_{\max}(\estcov) = \frac{1}{\lambda_{\min}(\estprec)}$, we have
\begin{align*}
\lambda_{\min}(\estprec)\le \frac{1}{1+c_1 \alpha \sqrt{n}},
\end{align*}
with probability at least $1-3p\exp\left(-c_2(\alpha \land 1) n\right)$, where $\alpha = p/n$. As $n\to \infty$, the upper bound approaches zero: the minimum eigenvalue of $\estprec$ poorly estimates that of $\trueprec$. We record this as a separate corollary.

\begin{corollary}
Suppose $\truecov = I_p$ and $p = \alpha n \ge 17$. Then 
\begin{align}
\vertiii{\estprec - \trueprec}_2
\ge 1-\lambda_{\min}(\estprec)
\ge \frac{1}{1+1/(c_1\alpha\sqrt{n})},
\end{align}
with probability at least $1-3\alpha ne^{-c_2n(\alpha \land 1)}$. Hence, $\estprec$ is inconsistent in the spectral norm as $n\to\infty$ and $p/n\to \alpha$.
\end{corollary}

\section{Discussion}\label{sec-discussion}

In this paper, we establish the possibility of tuning-free estimation of a large precision matrix $\trueprec$ based only on the knowledge that it is an $M$-matrix i.e., it has nonpositive off-diagonal entries. Our main contribution is to identify a loss---namely, the symmetrized Stein loss---in which $\estprec$ is both high-dimensionally consistent and minimax optimal. As the form \eqref{eq-goodness-of-fit} for the symmetrized Stein loss suggests, the quantity $\Ldiv(\aprec, \trueprec)$ is an average measure of closeness across all of the eigenvalues. The estimator $\estprec$ is inadequate, however, for estimating the extreme eigenvalues when $p$ is large relative to $n$, and our other main result establishes that $\estcov$ is minimax suboptimal in the spectral norm, even relative to the usual sample covariance matrix $S$. For the remainder of this section, we discuss some aspects that are naturally connected to our main results.

\textbf{Misspecification.} In practice, the assumption that all partial correlations are nonnegative may not hold exactly. \citet{slawski2015estimation} empirically evaluate the impact of misspecification on the estimator $\estprec$, defining the {\sl attractive part} $\aprec^\bullet\in \M$ of the population precision $\trueprec\not\in\M$ as the population analogue of the Bregman projection~\eqref{eq-log-det-div-min} with $S$ replaced by $\truecov$. Under the symmetrized Stein loss, a straightforward extension of Theorem~\ref{thm-symm-kl} shows that $\estprec$ targets the attractive part $\aprec^\bullet$ even under misspecification.

\begin{theorem}\label{thm-misspecification} Let $S = n^{-1}X^\TT X$ denote the sample covariance based on $X\in \R^{n\times p}$ with i.i.d.\ $\cn(0, \truecov)$ rows. Define the attractive part $\aprec^\bullet\in\M$ of the model as 
\begin{align*}
    \aprec^\bullet
    &:= \argmin_{\aprec \in \M} \left\{\<\aprec, \truecov\>
   - \log \det \aprec \right\}.
\end{align*}
For all $n\ge c_1\gamma^2(\acov^\bullet)\log p$, the MLE $\estprec$ defined in~\eqref{eq-log-det-div-min} satisfies
\begin{align*}
\Ldiv(\estprec, \aprec^\bullet)
&\le c_2\gamma(\acov^\bullet)\sqrt{\frac{\log p}{n}},
\end{align*}
with probability at least $1-c_3p^{-2}$. Here $c_1,c_2,c_3$ are universal positive constants.
\end{theorem}

\textbf{Estimating the correlation matrix.} 
One may also be interested, under the same nonnegative partial correlations assumption, in estimating the population correlation matrix $\truecor := D_{\truecov}^{-1/2}\truecov D_{\truecov}^{-1/2}$ and its inverse $\truepcor = (\truecor)^{-1} = D_{\truecov}^{1/2}\trueprec D_{\truecov}^{1/2}$ (here $D_{\truecov}$ denotes the diagonal matrix whose diagonal is equal to that of $\truecov$). It is natural to use $\estpcor := D_S^{1/2}\estprec D_S^{1/2}$ to estimate $\truepcor$. One can check that $\estpcor$ satisfies 
\begin{equation*}
   \estpcor = \argmin_{\apcor \in \M} \left\{\<\apcor, R\>
   - \log \det \apcor \right\}.
\end{equation*}
because the optimization problem is equivariant with respect to diagonal scaling \citep[see][Lemma 2.5]{lauritzen2019maximum}. 
The high-dimensional consistency result of Theorem~\ref{thm-symm-kl} also holds for $\estpcor$ as an estimator of the inverse correlation matrix $\truepcor$. This follows from an argument analogous to the proof of Theorem \ref{thm-symm-kl}, with the tail bound for $\|S - \truecov\|_{\infty}$ replaced by the corresponding tail bound on $\|R - \truecor\|_{\infty}$ \citep[see, e.g.,][Lemma 19]{sun2013sparse}.

\textbf{Non-Gaussian observations.} We state Theorems~\ref{thm-symm-kl} and~\ref{thm-misspecification} under the Gaussian assumption for simplicity and to remain consistent with other results in this paper. In general, the upper bound depends on the tail behavior of $\|S-\truecov\|_\infty$---see Lemma~\ref{lem-maximal-inequality}. A similar result holds when the rows of $X$ are i.i.d.\ with $\sigma$-sub-Gaussian components. As \citet{ravikumar2011high} note, estimators of the form~\eqref{eq-log-det-div-min} are motivated via maximum likelihood yet remain sensible for non-Gaussian $X$. For general $X$, the estimator $\estprec$ is motivated as a Bregman projection of $S$ with respect to the Stein loss.

\textbf{Modifying $\estprec$.} Although we focus on properties of the tuning-free estimator $\estprec$, additional processing such as thresholding $\estprec$ or pre-processing the sample covariance $S$ may produce an estimator that is high-dimensionally consistent in the spectral norm. The tuning-free covariance estimate $\estcov$ may also prove more useful for spectral analysis when the true covariance is a dense matrix. For instance, in the equicorrelation model where $\truecov$ has unit diagonal and every off-diagonal entry equal to $r \in (0,1)$, the entry-wise inequalities in \eqref{eq-dual} may introduce less bias.

\textbf{Related problems.} \cite{karlin1983m}, who pioneered the connection between $M$-matrices and $\mtptwo$, also considered repulsive models where the covariance matrix $\truecov\in \M$ has nonpositive off-diagonal, in which case all marginal and partial correlations are nonpositive. This also defines an interesting model class which may similarly simplify estimation in high-dimensional problems. Note, however, that the constraint set $\{\aprec : \aprec^{-1}\in \M\}$ of symmetric inverse-$M$ matrices is non-convex, presenting potential difficulties for maximum likelihood estimation.

\textbf{Connection to shape-restricted regression.} As a subset of the $p\times p$ symmetric positive-semidefinite matrices, the $M$-matrices $\M$ form a closed, convex cone determined only by sign constraints. The sign constraints on the precision matrix are analogous to a shape constraint in shape-restricted regression, enabling the use of likelihood techniques without explicit regularization. In particular, one can define the Bregman projection $\estprec$ of $S$ onto $\M$ \citep{slawski2015estimation, lauritzen2019maximum}. This work thus represents a first foray into the study of shape constraints for high-dimensional precision matrix estimation, inspired by results on regularization-free prediction in high-dimensional linear models via nonnegative least squares \citep{slawski2013non}. See \citet{groeneboom2014nonparametric} for a general introduction to shape-restricted regression and \citet{guntuboyina2018nonparametric} for a recent survey with a focus on risk bounds.

\section{Proofs}\label{sec-proofs}

\subsection{Proofs of Theorems~\ref{thm-symm-kl} and~\ref{thm-misspecification}}

We first introduce two lemmas needed in the proof of Theorem~\ref{thm-symm-kl}. Following previous results on sparse precision matrix estimation \citep[see, e.g.,][]{cai2011constrained, ravikumar2011high, sun2013sparse}, we rely on concentration of the entry-wise maximum deviation $\|S-\truecov\|_\infty = \max_{j,k} |S_{jk}-\truecov_{jk}|$ in the high-dimensional regime. A key technical tool in our analysis is the following lemma, which follows from an application of Bernstein's inequality. 

\begin{lemma}\label{lem-maximal-inequality} \cite[Lemma 6]{jankova2015confidence} Suppose $X\in \R^{n\times p}$ has i.i.d.\ $\cn(0, \truecov)$ rows and let $S = n^{-1}X^\TT X$. 
For any $t > 2$,
\begin{align*}
    \PP\left(
    \|S-\truecov\|_\infty \ge 2\|\truecov\|_\infty\left[\sqrt{\frac{2t\log p}{n}} +\frac{t\log p}{n}\right]
    \right)
    &\le \frac{2}{p^{t-2}}.
\end{align*}
\end{lemma}
\begin{proof} Let $\alpha = e_j$ and $\beta = e_k$ denote the standard basis vectors. Lemma 6 of \cite{jankova2015confidence} provides
\begin{align*}
    \PP\left(
    \alpha^\TT(S-\truecov)\beta \ge 2\|\truecov\|_\infty\left[\sqrt{\frac{2x}{n}} +\frac{x}{n}\right]
    \right)
    &\le 2e^{-x}.
\end{align*}
Taking a union bound over $j\le k$ and setting $x = \log p^t$ yields the claim.
\end{proof}

The next lemma records a distinctive property of $M$-matrices, corresponding to the fact that $M$-matrices are generalized diagonally dominant \citep{plemmons1977m}.

\begin{lemma}\label{lem-m-matrix-dd} Every $M$-matrix $\aprec\in \M$ satisfies $\|\aprec\|_1 := \sum_{i,j}|\aprec_{ij}|\le 2\emph{tr}(\aprec)$.
\begin{proof}
Since $\aprec$ is symmetric PSD, there are vectors $\theta_1,\dots,\theta_p$ such that $\aprec_{ij} = \langle\theta_i, \theta_j\rangle$. Moreover, since $\aprec$ has nonpositive off-diagonal entries, $\langle\theta_i, \theta_j\rangle\le 0$ for $i\ne j$. Hence
\[
\|\aprec\|_1
= \sum_{i}\|\theta_i\|_2^2 - \sum_{i\ne j}\langle\theta_i, \theta_j\rangle
= 2\sum_{i}\|\theta_i\|_2^2 - \left\|\sum_{i}\theta_i\right\|_2^2
\le 2\sum_{i}\|\theta_i\|_2^2 = 2\text{tr}(\aprec). \qedhere
\]
\end{proof}
\end{lemma}

An illustrative example is the one-parameter family of $p\times p$ symmetric matrices $A_x = (1-x)I_p + x{\bf 1}_p{\bf 1}_p'$ (where ${\bf 1}_p = \sum_{j=1}^p e_j$ is the all ones vector) with unit diagonal and every off-diagonal equal to $x$. Its eigenvalues are $1-x$ (with multiplicity $p-1$) and $1+(p-1)x$. Thus $A_x$ is PSD if and only if $x \in \left[-\frac{1}{p-1},1\right]$, whereas $A_x$ is an $M$-matrix if and only if $x \in \left[-\frac{1}{p-1},0\right]$. Finally, note $\|A_x\|_1 = p + p(p-1)|x|$ and $\text{tr}(A_x) = p$. This example shows Lemma~\ref{lem-m-matrix-dd} is tight. For general PSD matrices, the element-wise $\ell_1$-norm can be as large as $p$ times the trace, but for $M$-matrices it can be at most twice as large.

We are now ready to prove Theorem~\ref{thm-symm-kl}.

\begin{proof}[Proof of Theorem~\ref{thm-symm-kl}]
For any positive diagonal matrix $D\in \cd$, 
\begin{align*}
\Ldiv(\estprec(S), \trueprec)
&= \Ldiv(D\estprec(S)D, D\trueprec D) \\
&= \Ldiv(\estprec(D^{-1}SD^{-1}), D^{-1}\truecov D^{-1}),
\end{align*}
where the first step uses the fact that $\Ldiv$ is invariant under congruent transformations, and the second step uses the scale-invariance of the program~\eqref{eq-log-det-div-min}. With a sample covariance $S$ based on Gaussian observations, the loss $\Ldiv(\estprec, \trueprec)$ has the same distribution for covariance matrices of the form $\{D^{-1}\truecov D^{-1}\}_{D\in \cd}$. In particular, taking $D = D_{\truecov}^{1/2}$, we may assume without loss of generality that $\truecov$ is {\sl normalized}; i.e., $\truecov$ has unit diagonal or equivalently $\truecov$ equals the population correlation matrix $\truecor$. 

Let $f(\aprec) = \<\aprec, S\> - \log|\aprec|$. Since the estimator solves the constrained convex optimization problem $\estprec = \arg\min_{\aprec\in \M}f(\aprec)$, it is characterized by $\langle \nabla f(\estprec),\aprec - \estprec\rangle \ge 0$, for all $\aprec\in \M,$ where $\nabla f(\aprec) = S-\aprec^{-1}$. Hence
\[
\<S-\estcov, \trueprec - \estprec\> \ge 0.
\]
Rearranging yields the basic inequality
\begin{align*}
\Ldiv(\estprec, \trueprec) \le \frac{1}{2p}\left\langle S-\truecov, \trueprec-\estprec\right\rangle.
\end{align*}
Let $A := \|S-\truecov\|_\infty$. Using H\"older's inequality, we have:
\begin{align*}
\Ldiv(\estprec, \trueprec) 
&\le \frac{A}{2p}\left\| \trueprec-\estprec\right\|_1.
\end{align*}
Now applying the triangle inequality and Lemma~\ref{lem-m-matrix-dd} to the element-wise $\ell_1$-norm,
\begin{align*}
\Ldiv(\estprec, \trueprec) 
&\le \frac{A}{p}\left(\text{tr}(\trueprec)+\text{tr}(\estprec)\right).
\end{align*}
Since we have assumed without loss of generality that $\truecov=\truecor$,
\begin{align*}
\text{tr}(\estprec\truecov)
&= \text{tr}(\estprec) + \sum_{j\ne k} \estprec_{jk}\truecor_{jk}
\ge \left(1-\max_{j\ne k}\truecor_{jk}\right)\text{tr}(\estprec) \\
p=\text{tr}(\trueprec\truecov)
&= \text{tr}(\trueprec) + \sum_{j\ne k} \trueprec_{jk}\truecor_{jk}
\ge \left(1-\max_{j\ne k}\truecor_{jk}\right)\text{tr}(\trueprec),
\end{align*}
where we have again used Lemma~\ref{lem-m-matrix-dd}, along with the facts that $\estprec_{jk}$ and $\trueprec_{jk}$ are nonpositive for $j\ne k$ and $\truecov\ge 0$ entry-wise \citep[see][Chap. 6, Property $N_{38}$]{berman1994nonnegative}. Combining the last three displays and using the characterization of $\gamma(\truecov)$ in \eqref{eq-gamma-def}, we get
\begin{align*}
\Ldiv(\estprec, \trueprec) 
&\le \frac{\gamma(\truecov) A}{p}\left(p+\text{tr}(\estprec\truecov)\right) \\ 
&\le \gamma(\truecov) A\left(3+2\Ldiv(\estprec, \trueprec)\right).
\end{align*}
On the event $E=\left\{2\gamma(\truecov) A \le \frac{1}{2}\right\}$, we have $\Ldiv(\estprec, \trueprec) \le 6\gamma(\truecov) A$. Applying Lemma~\ref{lem-maximal-inequality} with $t=4$, the event $E' = \left\{A\le 2\sqrt{\frac{8\log p}{n}} + \frac{8\log p}{n}\right\}$ occurs with probability at least $1-2/p^2$. 

To guarantee $E'\subset E$, we require
\[
2\sqrt{\frac{8\log p}{n}} + \frac{8\log p}{n}
\le \frac{1}{4\gamma(\truecov)},
\]
which is equivalent to 
\[
\frac{\log p}{n}
\le 2+\frac{1}{4\gamma(\truecov)} -  \sqrt{4+\gamma^{-1}(\truecov)}.
\]
Using $\gamma(\truecov) \ge 1$, it is straightforward to check that the right hand side above is at least $\frac{1}{72\gamma^2(\truecov)}$. Hence, as long as $n\ge 72\gamma^2(\truecov)\log p$, 
\[
\Ldiv(\estprec, \trueprec)
\le 6\gamma(\truecov)\left(2\sqrt{\frac{8\log p}{n}} + \frac{8\log p}{n}\right)
\]
with probability at least $1-2/p^2$. Since $\gamma(\truecov)\ge 1$, the $\sqrt{\frac{8\log p}{n}}$ dominates the $\frac{8\log p}{n}$ term. In particular, we have $\sqrt{\frac{8\log p}{n}} \le \frac{1}{3}$, so $\Ldiv(\estprec, \trueprec) \le 28\gamma(\truecov)\sqrt{\frac{2\log p}{n}}$ with probability at least $1-2/p^2$.
\end{proof}


\begin{proof}[Proof of Theorem~\ref{thm-misspecification}] Since the attractive part $\aprec^\bullet$ is an $M$-matrix, from the first order optimality conditions for $\estprec$,
\[
\<S-\estcov, \aprec^\bullet - \estprec\>
\ge 0.
\]
Using the first order optimality conditions for $\aprec^\bullet$ and the fact that $\estprec\in \M$, 
\[
\<\truecov-\acov^\bullet, \estprec - \aprec^\bullet\>
\ge 0.
\]
Adding these and rearranging yields the basic inequality
\[
\Ldiv(\aprec^\bullet, \estprec)
\le \frac{1}{2p}\<S-\truecov, \aprec^\bullet - \estprec\>.
\]
The remainder of the proof proceeds as the proof of Theorem~\ref{thm-symm-kl}, substituting $\trueprec$ with $\aprec^\bullet$.
\end{proof}

\subsection{Proof of Theorem~\ref{thm-minimax-rate}}\label{ssec-minimax-proofs}

\begin{proof}[Proof of Theorem~\ref{thm-minimax-rate}] As in \citet[Proof of Theorem 4.1]{cai2016estimating}, we consider precision matrices of the form 
\begin{align}\label{eq-cai-construction}
\aprec=\left[
\begin{array}{cc}
I_{\lceil p/2\rceil} & {\eps A} \\ 
\eps A^\TT & {I_{\lfloor p/2\rfloor}}
\end{array}\right],
\end{align}
where $A$ is a sparse binary matrix with $k$ nonzero entries per row and at most $2k$ nonzero entries per column, for some positive integer $k$ and some $\eps$ to be chosen later. As long as $\eps < 0$ and $2k|\eps| < 1$, the matrix $\Theta$ is a diagonally dominant $M$-matrix. Its inverse is given by the Neumann series
\begin{align*}
\acov = \aprec^{-1} 
&=\sum_{m=0}^\infty (-\eps)^m\left[
\begin{array}{cc}
0 & A \\ 
A^\TT & 0
\end{array}\right]^m  \\
&=\sum_{m=0}^\infty \eps^{2m}\left[
\begin{array}{cc}
(AA^\TT)^m & -\eps A(A^\TT A)^m \\ 
-\eps A^\TT(AA^\TT)^m & (A^\TT A)^m
\end{array}\right]
\end{align*}
From the last display, it is clear that $D_\acov \ge I_p$, so $\max_{j\ne k} \acor_{jk} \le \max_{j\ne k} \acov_{jk}$ where $\acor$ is the correlation matrix corresponding to $\acov$. Furthermore, by triangle-inequality, the largest off-diagonal entry of the top left diagonal block is at most 
\begin{align*}
\left\|\sum_{m=0}^\infty\eps^{2m}(AA^\TT)^m\right\|_{\infty, \text{off}}
&\le \sum_{m=1}^\infty \eps^{2m}\left\|(AA^\TT)^m\right\|_{\infty, \text{off}},
\end{align*}
where we use that the first term $m=0$ has zero off-diagonal. This yields
\[
\left\|(AA^\TT)^m\right\|_{\infty, \text{off}} \le \vertiii{(AA^\TT)^m}_{2}
\le (2k)^{2m}.
\]
By similar bounds on the other blocks of $\acov$, it can be shown that
\[
\max_{j\ne k} \acor_{jk} 
\le \max_{j\ne k} \acov_{jk}
\le \frac{2k|\eps|}{1-(2k\eps)^2}.
\]
A simple sufficient condition to guarantee $\gamma(\acov)\le \gamma$ is thus $4k|\eps|\le (1-\gamma^{-1})\land \frac{1}{2}.$

By the Ger\u{s}gorin circle theorem, the spectrum of $\Theta$ lies in the range $[0,2]$. Further constraining the supremum in~\eqref{eq-minimax} to $\lambda_{\max}(\aprec)\le 2$, by \citet[Eq.~(54]{cai2012optimal}), we have:
\[
\inf_{\breve\aprec}\sup_{\substack{\aprec\in \M \\ \gamma(\acov)\le \gamma}} \E\Ldiv(\breve\aprec, \aprec)
\ge \frac{1}{4}\inf_{\breve\aprec}\sup_{\substack{\aprec\in \M \\ \gamma(\acov)\le \gamma \\ \lambda_{\max}(\aprec)\le 2}}\EE\frac{\|\breve\aprec-\aprec\|_F^2}{p},
\]
so it suffices to lower bound the minimax rate in the Frobenius norm.

Now let $\ca$ denote the set of all $\lceil p/2\rceil \times \lfloor p/2\rfloor$ binary matrices with $k$ nonzero entries per row and at most $2k$ nonzero entries per column, and $\cb = \{0,1\}^{\lceil p/2\rceil}$. Finally, let $e$ denote a vector of ones of length $\lfloor p/2\rfloor$. Given $A\in \ca$ and $b\in \cb$, the matrix $(b\otimes e)\circ A$ has the same shape as $A$, where the $j^\text{th}$ row is nonzero if and only if $b_j=1$. Let 
\[
\cf = \left\{\aprec_{A,b}=\left[
\begin{array}{cc}
I_{\lceil p/2\rceil} & {\eps (b\otimes e)\circ A} \\ 
\eps (b^\TT\otimes e^\TT)\circ A^\TT & {I_{\lfloor p/2\rfloor}}
\end{array}\right] :
A\in \ca, b\in \cb
\right\}.
\]
As we have shown, $\cf\subset\{\aprec\in \M : \gamma(\acov)\le \gamma, \lambda_{\max}(\aprec)\le 2\}$. By \cite[][Lemma~3]{cai2012optimal}
\[
\inf_{\breve\aprec}\max_{\aprec\in \cf}\EE\frac{\|\breve\aprec-\aprec\|_F^2}{p}
\ge \frac{1}{32}
\left[\min_{\substack{A,A'\in\ca, b,b'\in \cb \\ b\ne b'}}\frac{\|\aprec_{A,b} - \aprec_{A',b'}\|_F^2}{H(b,b')}\right]
\left[\min_{1\le j\le \lceil p/2\rceil}\|\bar{P}_{j, 0}\land \bar{P}_{j, 1}\|\right],
\]
where $H$ denotes the Hamming distance and $\|\bar{P}_{j, 0}\land \bar{P}_{j, 1}\|$ denotes the total variation affinity between the measures $\bar{P}_{j, 0}$ and $\bar{P}_{j, 1}$, where $\bar{P}_{j, i}$ is the uniform mixture over $\cn(0, \aprec_{A,b}^{-1})$ over all $A\in \ca$ and all $b\in \cb$ such that $b_j=i$.

For the first term, fix $A,A'$ and $b\ne b'$. For $j$ such that $b_j\ne b_j'$, if say $b_j=0$, the $j^\text{th}$ row of $(b\otimes e)\circ A$ is zero and the $j^\text{th}$ row of $(b'\otimes e)\circ A'$ has $k$ nonzero entries. Hence 
\begin{align*}
    \min_{\substack{A,A'\in\ca, b,b'\in \cb \\ b\ne b'}}\frac{\|\aprec_{A,b} - \aprec_{A',b'}\|_F^2}{H(b,b')}
    &\ge \min_{\substack{A,A'\in\ca, b,b'\in \cb \\ b\ne b'}}\frac{2\sum_{j : b_j\ne b_j'}k\eps^2}{H(b,b')} = 2k\eps^2.
\end{align*}
In particular, we have shown 
\begin{align*}
    \inf_{\breve\aprec}\sup_{\substack{\aprec\in \M \\ \gamma(\acov)\le \gamma}} \E\Ldiv(\breve\aprec, \aprec)
    \ge ck\eps^2 \min_{1\le j\le \lceil p/2\rceil}\|\bar{P}_{j, 0}\land \bar{P}_{j, 1}\|.
\end{align*}

Finally, the same argument of \citep[][proof of Lemma 4.5]{cai2016estimating} with $\eps = c'\sqrt{\frac{\log p}{n}}$ can be used to show $\min_{1\le j\le \lceil p/2\rceil}\|\bar{P}_{j, 0}\land \bar{P}_{j, 1}\| \ge c'' > 0$, yielding
\[
    \inf_{\breve\aprec}\sup_{\substack{\aprec\in \M \\ \gamma(\acov)\le \gamma}} \E\Ldiv(\breve\aprec, \aprec)
    \ge cc''k\eps^2 = c_\gamma\eps. \qedhere
\]
\end{proof}

\subsection{Proof of Proposition \ref{prop-minimaxdiagpsd}}

\begin{proof}[Proof of Proposition \ref{prop-minimaxdiagpsd}] Let $\tilde\Sigma^{\text{DIAG}} = c\cdot D_S$. Since $S_{11} = \frac{1}{n}\sum_{i=1}^n X_{i1}^2$ for $X_{i1}\simiid \cn(0, \truecov_{11})$,
\begin{align*}
\mathbb{E}\left[\Ldiv(\tilde\aprec^{\text{DIAG}}, \trueprec)\right]
&= \frac{1}{2}\mathbb{E}\left[\frac{\truecov_{11}}{cS_{11}} + \frac{cS_{11}}{\truecov_{11}} - 2\right]
= \frac{1}{2}\left[\frac{1}{c}\frac{n}{n-2} + c - 2\right].
\end{align*}
The minimum is achieved at $c = \sqrt{\frac{n}{n-2}}$, but taking $c = 1$ suffices to prove the minimax rate~\eqref{eq-minimax-diag} is upper bounded by $\frac{C}{n}$.

Now consider a prior $G$ on $\cd$ over which the components $\trueprec_{jj}$ are i.i.d. Lower bound the minimax risk by the Bayes risk with respect to $G$: 
\begin{align*}
\inf_{\hat\Sigma}\sup_{\truecov\in \cd} \E \Ldiv(\hat\Sigma, \truecov)
&\ge \inf_{\hat\Sigma} \E_G\Ldiv(\hat\Sigma, \truecov) \\
&= \inf_{\hat\Sigma_{11}} \E_G\Ldiv(\hat\Sigma_{11}, \truecov_{11}).
\end{align*}
If $G = [\text{Gamma}(a,b)]^{\otimes n}$, such that $\trueprec_{jj}\simiid \text{Gamma}(a,b)$ under $G$, then combining with the likelihood we have:
\begin{align*}
S_{11}\mid \trueprec_{11} \sim \text{Gamma}(a,b).
\end{align*} 
By conjugacy, the posterior is readily seen to be 
\begin{align*}
\trueprec_{11} \mid S_{11}=s \sim \text{Gamma}\left(a+\frac{n}{2},b+\frac{ns}{2}\right).
\end{align*} 
Thus, for $n > 2$, 
\begin{align*}
\E_G\left[\Ldiv(\frak{d}, \truecov_{11})\mid S_{11} = s\right]
&= \frac{1}{2}\E_G\left[\frak{d}\trueprec_{11} + \frac{\truecov_{11}}{\frak{d}} -2\mid S_{11} = s\right] \\
&= \frac{\frak{d}}{2}\frac{a+n/2}{b+ns/2} + \frac{1}{2\frak{d}}\frac{b+ns/2}{a+n/2-1} -1.
\end{align*}
This is minimized at $\frak{d}^* = \frac{b+ns/2}{\sqrt{(a+n/2)(a+n/2-1)}}$, giving a Bayes risk of 
\begin{align*}
\E_G\left[\Ldiv(\frak{d}^*, \truecov_{11})\right]
&= \sqrt{\frac{a+n/2}{a+n/2-1}}-1.
\end{align*}
Letting $a\downarrow 0$, we find 
\begin{align*}
\inf_{\hat\Sigma}\sup_{\truecov\in \cd} \E \Ldiv(\hat\Sigma, \truecov)
&\ge \sqrt{1 + \frac{2}{n-2}}-1 \\
&= \frac{1}{n-2} + o(n^{-1}),
\end{align*}
as $n\to \infty$. This proves the minimax rate~\eqref{eq-minimax-diag} on $\cd$.

To prove the lower bound~\eqref{eq-minimax-psd} on $\psd$, place an inverse Wishart prior $\truecov \sim \cw^{-1}(\Sigma_0, \nu)$ on the covariance matrix. By conjugacy, 
\begin{align*}
\truecov\mid S \sim \cw^{-1}(\truecov, \Sigma_0 + nS, \nu +n).
\end{align*}
As long as $\nu+n > p+1$, the posterior loss can be written in closed form as
\begin{align*}
\E\left[\Ldiv(\hat\aprec, \trueprec)\mid S\right]
&= \frac{1}{2p}\left[(\nu + n)\text{tr}(\hat\Sigma (\Sigma_0+nS)^{-1}) + \frac{\text{tr}(\hat\aprec(\Sigma_0 + n S))}{\nu + n - p- 1} - 2p\right],
\end{align*}
which is minimized at $\hat\aprec = \sqrt{(\nu+n)(\nu+n-p-1)}(\Sigma_0 + nS)^{-1}$, yielding a Bayes risk of 
\begin{align*}
\E\left[\Ldiv(\hat\aprec, \trueprec)\mid S\right]
&= \sqrt{\frac{\nu+n}{\nu+n-p-1}}-1,
\end{align*}
independent of $\Sigma_0$. Setting $\nu = p+1$,
\begin{align*}
\inf_{\hat\aprec = \hat\aprec(S)}\sup_{\trueprec\succeq 0} \E\Ldiv(\hat\aprec, \trueprec) 
\ge \sqrt{1+\frac{p+1}{n}}-1,
\end{align*}
Finally, use $\sqrt{1+x}-1 \ge (\sqrt{2}-1)\left(x\land \sqrt{x}\right)$ for any $x \ge 0$.
\end{proof}

\subsection{Proof of Theorem \ref{thm-instance-lower-bound}}

\begin{proof}[Proof of Theorem \ref{thm-instance-lower-bound}] 
This proof uses Theorem \ref{thm-inconsistency} which is proved in the next subsection. Since $\truecov\in \cd$, as in the proof of Theorem~\ref{thm-symm-kl} we have
\begin{align*}
\Ldiv(\estprec(S), \trueprec)
&= \Ldiv\left(\estprec(D_{\truecov}^{-1/2}SD_{\truecov}^{-1/2}), I_p\right).
\end{align*}
In particular, due to scale invariance of both the estimator and the loss, the symmetrized Stein loss $\Ldiv(\estprec, \trueprec)$ has the same distribution for all diagonal matrices $\truecov\in \cd$. We thus assume with no loss of generality that $\truecov=I_p$.

Let $f(t) = t+t^{-1}-2$ for $t > 0$. 
By \eqref{eq-goodness-of-fit} and nonnegativity of the function $f$,
\begin{align*}
\Ldiv(\estprec, I_p)
= \frac{1}{p}\sum_{j=1}^p f(\lambda_{j}(\estprec))
\ge \frac{f(\lambda_{\max}(\estprec))}{p}.
\end{align*}
For $t > 1$, $f'(t) > 0$, so by Theorem~\ref{thm-inconsistency},
\begin{align*}
\Ldiv(\estprec, I_p)
&\ge \frac{1}{p}f\left(1+c_1\frac{p}{\sqrt{n}}\right) 
= \frac{c_1}{\sqrt{n}}\left[1 - \frac{1}{1+c_1\frac{p}{\sqrt{n}}}\right],
\end{align*}
with probability at least $1-3p\exp\left(-c_2(n\land p)\right)$. If $c_1p \ge \sqrt{n}$, this implies $\Ldiv(\estprec, I_p) \ge \frac{c_1}{2\sqrt{n}}$, completing the proof. 
\end{proof}

\subsection{Proof of Theorem~\ref{thm-inconsistency}}

The most technically involved part of the proof is a lower bound on the row sums of the positive part $S_+$ of the sample covariance matrix, which we include as a separate lemma.

\begin{lemma}\label{lem-fixed-i} Under the conditions of Theorem~\ref{thm-inconsistency}, 
\begin{align*}
\sum_{j=1}^p (S_+)_{pj}
\ge 1+c_0 \frac{p}{\sqrt{n}},
\end{align*}
with probability at least $1-3\exp\left(-c_1(n\land p)\right)$, for some universal positive constants $c_0,c_1$.
\end{lemma}

We  give the proof of Theorem \ref{thm-inconsistency} assuming the above lemma and then prove the lemma subsequently. 

\begin{proof}[Proof of Theorem~\ref{thm-inconsistency}] Since $\estcov$ is an inverse $M$-matrix, it is entry-wise nonnegative; i.e., $\estcov\ge 0$. Combining this with the first constraint $\estcov\ge S$ in the dual formulation \eqref{eq-dual}, we have that $\estcov\ge S_+\ge 0$, where $S_+$ is the entry-wise positive part of the sample covariance matrix $S$. The Perron-Frobenius theorem \cite[Corollary 1.5]{berman1994nonnegative} gives 
\begin{align*}
\lambda_{\max}(\estcov)\ge \lambda_{\max}(S_+).
\end{align*}
Thus, we want to show that $\lambda_{\max}(S_+)$ is more severely biased than $\lambda_{\max}(S)$. To this end, we apply another standard result from the spectral theory of nonnegative matrices \cite[Theorem 2.35]{berman1994nonnegative}:
\begin{align*}
\lambda_{\max}(S_+)
\ge \min_k\sum_{j} (S_+)_{jk}.
\end{align*}
By Lemma~\ref{lem-fixed-i}, $\sum_{j} (S_+)_{jk}\ge 1+c_0 \frac{p}{\sqrt{n}}$ with probability at least $1-3e^{-c_1(n\land p)}$ for each fixed $k$, so by a union bound, 
\begin{align*}
\min_k\sum_{j} (S_+)_{jk}
\ge 1+c_0 \frac{p}{\sqrt{n}}
\end{align*}
with probability at least $1-3pe^{-c_1(n\land p)}$. Combining the last three displays gives the desired lower bound on $\lambda_{\max}(\estcov)$.
\end{proof}

We now prove the key lemma on the row sums of $S_+$.

\begin{proof}[Proof of Lemma~\ref{lem-fixed-i}] 
For $u > 0$, write
  \begin{align*}
    \PP \left\{\sum_{j=1}^p (S_+)_{pj} \le 1+u  \right\} &\leq \PP \left\{S_{pp} \leq 1 - u \right\} + \PP \left\{\sum_{j < p} (S_{pj})_+ \le 2u \right\}. 
  \end{align*}
  To bound the first term, note that $n S_{pp} \sim \chi^2_n$ and the following standard chi-squared lower tail bound (see e.g., \citet[inequality (4.4)]{laurent2000adaptive}):
  \begin{equation}\label{eq-chisquare-tail}
      \PP \left\{\frac{\chi^2_n}{n} \leq 1 - u \right\} \leq \exp \left(\frac{-nu^2}{4} \right)  
  \end{equation}
  gives 
  \begin{align}\label{coml}
  \PP \left\{S_{pp} \leq 1 - u \right\}\le \exp \left(-\frac{nu^2}{4} \right). 
  \end{align}
  To bound the second term, notice that conditionally on $X_{ip}, i=1,\dots, n$,
  \begin{align*}
    S_{pj}, j = 1, \dots, p-1 \bigg| X_{ip}, i = 1, \dots, n \overset{\text{i.i.d}}{\sim} N\left(0, \frac{1}{n^2} \sum_{i=1}^n X_{ip}^2 \right). 
  \end{align*}
  Thus, conditionally on $X_{ip}, i = 1, \dots, n$, we can write $S_{pj} = AZ_j$ for $j = 1, \dots, p-1$ where
  \begin{align*}
    A^2 := \frac{1}{n^2} \sum_{i=1}^n X_{ip}^2 ~~ \text{ and } ~~ Z_1, \dots, Z_{p-1} \overset{\text{i.i.d}}{\sim} N(0, 1). 
  \end{align*}
  We can therefore write (using the notation $\PP^|$ for probability conditioned on $X_{ip}, i = 1, \dots, n$)
  \begin{align*}
    \PP^| \left\{\sum_{j < p} (S_{pj})_+ \le 2u \right\} 
    &=  \PP^| \left\{ \sum_{j < p} (Z_j)_+ \leq \frac{2u}{A} \right\} \\
    &= \PP^| \left\{ \frac{1}{p-1}\sum_{j < p}\left( (Z_j)_+ - c \right) \leq \frac{2u}{(p-1)A} - c\right\},
  \end{align*}
  where $c := \E (Z_1)_+ = (2\pi)^{-1/2}$  is a universal constant. We now note that
  \begin{align*}
    (z_1, \dots, z_{p-1}) \mapsto \frac{1}{p-1}\sum_{j < p}(z_j)_+
  \end{align*}
  is a Lipschitz function with Lipschitz constant $(p-1)^{-1/2}$. Thus by the usual concentration inequality for Lipschitz functions of Gaussian random vectors~\citep[see, e.g.,][Theorem 2.26]{wainwright2019high}, we obtain
  \begin{align*}
    \PP^| \left\{ \frac{1}{p-1}\sum_{j < p}\left( (Z_j)_+ - c \right) \leq \frac{2u}{(p-1)A} - c\right\} \leq \exp \left(-\frac{(p-1)}{2} \left(c - \frac{2u}{(p-1)A} \right)^2\right),
  \end{align*}
  assuming that $c > 2u/(A(p-1))$. In particular, for $c > 4u/(A(p-1))$, we get
  \begin{align*}
    \PP^| \left\{ \frac{1}{p-1}\sum_{j < p}\left( (Z_j)_+ - c \right) \leq \frac{2u}{(p-1)A} - c\right\} \leq \exp \left(-\frac{(p-1)c^2}{8} \right).
  \end{align*}
  We have thus proved
  \begin{align*}
    \PP^| \left\{\sum_{j < p} (S_{pj})_+ \le 2u \right\} \leq \exp \left(-\frac{(p-1)c^2}{8} \right) + I \left\{c \leq 4u/(A(p-1)) \right\}. 
  \end{align*}
  Taking an expectations on both sides of this expression, we obtain
  \begin{align*}
    \PP \left\{\sum_{j < p} (S_{pj})_+ \le 2u \right\} \leq \exp \left(-\frac{(p-1)c^2}{8} \right) + \PP \left\{A \leq \frac{4u}{(p-1) c} \right\}.
  \end{align*}
  Note now that $n^2 A^2 \sim \chi^2_n$ and thus
  \begin{align*}
    \PP \left\{A \leq \frac{4u}{(p-1) c} \right\} = \PP \left\{\frac{\chi^2_n}{n} - 1 \leq \frac{16u^2n}{(p-1)^2 c^2} - 1 \right\}. 
  \end{align*}
  We now make the choice $u = \frac{(p-1) c}{4\sqrt{2} \sqrt{n}}$,
  which gives (via \eqref{eq-chisquare-tail})
  \begin{align*}
    \PP \left\{A \leq \frac{4u}{(p-1) c} \right\} = \PP \left\{\frac{\chi^2_n}{n} - 1 \leq \frac{-1}{2} \right\} \leq \exp \left(-\frac{n}{16} \right). 
  \end{align*}
  We have thus proved
  \begin{align*}
     \PP \left\{\sum_{j < p} (S_{pj})_+ \le \frac{(p-1) c}{2\sqrt{2} \sqrt{n}}     \right\} \leq \exp \left(-\frac{(p-1)c^2}{8} \right) + \exp \left(-\frac{n}{16} \right). 
  \end{align*}
  Combining this with \eqref{coml} and using $c= (2\pi)^{-1/2}$, we obtain
  \[
\PP \left\{\sum_{j=1}^p (S_+)_{pj} \le 1+\frac{(p-1)}{8 \sqrt{\pi n}}   \right\} 
\leq     \exp \left(-\frac{(p-1)^2}{256\pi} \right) + \exp \left(-\frac{p-1}{16\pi} \right) + \exp \left(-\frac{n}{16} \right).
  \]
For $p  \ge 17$ the first term is of lower order; i.e., $\exp \left(-\frac{(p-1)^2}{256\pi} \right) \le \exp \left(-\frac{p-1}{16\pi} \right)$.
\end{proof}

\section*{Acknowledgements}

We would like to thank Martin Wainwright, Peter Bickel, and Eli Ben-Michael for valuable discussion.




\nocite{mardiamultivariate}
\bibliographystyle{imsart-nameyear}
\bibliography{citations}       

\end{document}